\documentclass[]{cDSS2e}

\usepackage[active]{srcltx}

\begingroup\makeatletter\ifx\SetFigFont\undefined%
\gdef\SetFigFont#1#2#3#4#5{%
  \reset@font\fontsize{#1}{#2pt}%
  \fontfamily{#3}\fontseries{#4}\fontshape{#5}%
  \selectfont}%
\fi\endgroup%

\begin{document}

\doi{10.1080/1745973YYxxxxxxxx}
 \issn{1745-9745} \issnp{1745-9737}

  \jnum{00} \jyear{2008} \jmonth{January}

\markboth{Taylor \& Francis and I.T. Consultant}{Dynamical Systems}

\title{{\itshape Global dynamics of a family of \\3-D Lotka--Volterra Systems} }

\author{A. C. Murza$^{\rm a}$$^{\ast}$\thanks{$^\ast$Corresponding author. Email: amurza@ifisc.uib-csic.es
\vspace{6pt}} and A. E. Teruel$^{\rm b}$\\\vspace{9pt}  $^{\rm a}${\em{IFISC, Institut de F\'{\i}sica Interdisciplinaria i Sistemes Complexes (CSIC - UIB), Universitat de les Illes Balears, Crta. de Valldemossa km. 7.5, 07122 Palma de Mallorca, Spain.}};\\ 
$^{\rm b}${\em{Departament de Matem\`{a}tiques i Inform\`atica, Universitat de les Illes
Balears, Crta. de Valldemossa km. 7.5, 07122 Palma de Mallorca, Spain.}}\\\vspace{9pt}\received{v1 released July 2009} }

\maketitle

\bigskip
\begin{abstract}

In this paper we analyze the flow of a family of  three dimensional Lotka-Volterra systems restricted to an invariant and bounded region.
The behaviour of the flow in the interior of this region is simple: either every orbit is a periodic orbit or they move from one boundary to another. Nevertheless the complete study of the limit sets in the boundary allows to understand the bifurcations which take place in the region as a global bifurcation that we denote by focus--center--focus bifurcation.

\bigskip


\begin{classcode}
34C23;34C30;34C15;37G35
\end{classcode}\bigskip

\end{abstract}

\section{Introduction}

Consider a closed chemical system composed of four coexisting chemical species denoted by $X,Y,Z$ and $V,$ which represent four possible states of a macromolecule operating in a reaction network far from equilibrium. As discussed by Wyman \cite{W75}, such a reaction can be modeled as a ``turning wheel'' of one--step transitions of the macromolecule, which circulate in a closed reaction path involving the four possibles states. The turning wheels have been proposed by Di Cera et al. \cite{DPW88} as a generic model for macromolecular autocatalytic interactions. 

While Di Cera's model considers unidirectional first order interactions, Murza et al. in \cite{MOD20} consider a closed sequence of chemical equilibria. In their approach the reaction rates are defined as functions of the time dependent product concentrations, multiplied by their reaction rate constants. This type of reaction rates has been introduced in Wyman's original paper \cite{W75}. 

Following the closed sequence of chemical equilibria in \cite{MOD20}, the autocatalytic chemical reactions between $X,Y,Z$ and $V$ (see Figure \ref{fig:reaction}) 
\begin{figure}
\begin{center}
\includegraphics{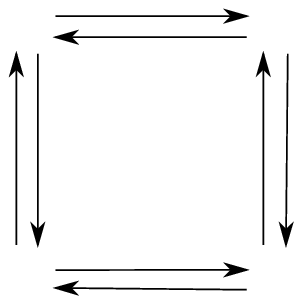}
\begin{picture}(0,0)
\put(-0.45,0){$Z$}
\put(-3,0){$V$}
\put(-0.4,2.65){$Y$}
\put(-3,2.65){$X$}
\put(-1.8,3.1){$k_{11} y$}
\put(-1.8,2.2){$k_{12} x$}
\put(-1.8,0.45){$k_{32} z$}
\put(-1.8,-0.35){$k_{31} v$}
\put(-3.7,1.45){$k_{41}x$}
\put(-2.6,1.45){$k_{42} v$}
\put(-1.2,1.45){$k_{22}y$}
\put(-0,1.45){$k_{21}z$}
\end{picture}
\vspace{0.5cm}
\caption{\label{fig:reaction}Closed sequence of chemical equilibria between $X,Y,Z$ and $V$.}
\end{center}
\end{figure}
are governed by the following 4--parameter family of nonlinear differential equations
\begin{equation}\label{4Dequation}
\begin{array}{l}
\dot{x}=x(k_1 y - k_4 v),\\
\dot{y}=y(k_2 z -k_1 x),\\
\dot{z}=z(k_3 v -k_2 y),\\
\dot{v}=v(k_4 x -k_3 z).
\end{array}
\end{equation}
Functions $x(t),y(t),z(t)$ and $v(t)$ are concentrations at time $t$ of the chemical species $X,Y,Z,V$ respectively. The parameters $k_i=k_{i2}-k_{i1}$ for $i=1,2,3,4$ are differences of pairs of reaction rate constants corresponding to each chemical equilibrium. It can be easily seen that system (\ref{4Dequation}) is identical to Di Cera's model restricted to $n=4,$ see equation (7) in \cite{DPW88}. In that work, Di Cera claims that this family exhibits self sustained and conservative oscillations only when the parameter $\mathbf{k}=(k_1,k_2,k_3,k_4)$ is in the three dimensional manifold $\mathcal{S}=\left\{\mathbf{k}\in \mathbb{R}^4 \setminus \{\mathbf{0}\}: k_1 k_3-k_2 k_4=0\right\}$.

Assuming that the conservation of mass $x+y+z+v=1$ applies to the macromolecular system (\ref{4Dequation}), its kinetic behaviour is described by the three--dimensional system of polynomial differential equations
\begin{equation}\label{3DLV}
\left\{
\begin{array}{l}
\dot{x}=x(k_1 y -k_4 (1-x-y-z)),\\
\dot{y}=y(k_2 z - k_1 x),\\
\dot{z}=z(-k_2 y + k_3 (1-x-y-z)),
\end{array}
\right.
\end{equation}
restricted to the flow--invariant bounded region $\mathcal{T}=\{x \geq 0, y\geq 0, z\geq 0, x+y+z \leq 1\}.$ 

The system (\ref{3DLV}) is a particular case of the class of three--dimensional Lotka--Volterra  systems (LVS)
\[
 \dot{x}_i=x_i\left( a_i + \sum_{i=1}^3 b_{ij}x_j\right),\quad i=1,2,3,
\]
which has been extensively studied starting with the pioneer works of Lotka \cite{L20} and Volterra \cite{V31}. These systems have multiple applications in biochemistry. For instance, enzyme kinetics \cite{W75} , circadian clocks \cite{MH75} and genetic networks \cite{AKJ08,CDL04}  often produce sustained oscillations modeled with LVS.

Solutions of LVS cannot, in general, be written in terms of elementary functions. So that the search for invariant manifolds, first integrals or/and integrability conditions can be useful to the analysis of the flow. This approach has found increasing popularity over the last few years after the works of Christopher and Llibre \cite{C94,CL99,CLS97}, which are based on the Darboux's theory of integrability, see for instance \cite{CL00,C00,LR00} and references therein. Unfortunately for systems of dimension greater than 2 the behaviour of the flow is not entirely known, even when the system is integrable.
Of course in the case of non--integrable LVS the lack of knowledge is higher and other tools are required. Some results concerning the existence of limit cycles for special parameter sets can be found in \cite{DZ98,Z93,Z96,H88}. Additional results about the number of limit cycles which can appear after perturbation are presented in \cite{MZ05}.

In this paper we deal with the global analysis of the flow of system (\ref{3DLV}) restricted to the region $\mathcal{T}.$ Note that the boundary $\partial \mathcal{T}$ of the region is a three dimensional simplex which is invariant by the flow. This boundary is formed by the union of the following invariant subsets: 
the invariant faces 
$\mathcal{X}=\{(0,y,z): y>0,z>0,y+z< 1\},$
$\mathcal{Y}=\{(x,0,z): x>0,z>0,x+z<1\},$
$\mathcal{Z}=\{(x,y,0): x>0,y>0,x+y<1\}$ and
$\Sigma=\{(x,y,z): x>0,y>0,z>0, x+y+z =1\};$
and the invariant edges
$\mathcal{R}_{yz}=\{(x,0,0): 0<x<1\},$
$\mathcal{R}_{xz}=\{(0,y,0): 0\leq y \leq 1\},$
$\mathcal{R}_{xy}=\{(0,0,z): 0<z<1\},$
$\mathcal{R}_{px}=\{(0,y,1-y): 0<y<1\},$ 
$\mathcal{R}_{py}=\{(x,0,1-x): 0\leq x \leq 1\},$ and
$\mathcal{R}_{pz}=\{(x,1-x,0): 0<x<1\}.$
We remark that the edges $\mathcal{R}_{xz}$ and $\mathcal{R}_{py}$ are closed segments formed by singular points. 

In order to make easier the analysis we consider the following subsets in the parameter space: 
$\mathcal{S}^-=\{\mathbf{k}\in \mathbb{R}^4: k_1 k_3-k_2 k_4<0\},$ $\mathcal{S}^+=\{\mathbf{k}\in \mathbb{R}^4: k_1 k_3-k_2 k_4 >0\},$ 
$\mathcal{NZ}=\{\mathbf{k}\in \mathbb{R}^4: k_1 k_2 k_3 k_4 \neq 0\}$ and 
$\mathcal{PS}=\{\mathbf{k}\in \mathbb{R}^4: k_1 k_2 >0, k_1 k_3>0, k_1 k_4 >0\}.$
We note that $\mathcal{S}^-$ and $\mathcal{S}^+$ together with $\mathcal{S}$ (defined above) form a partition of the parameter space $\mathbb{R}^4.$ We also note that the parameter set $\mathcal{PS}$ is a subset of $\mathcal{NZ}.$

The main result of the paper is summarised in the following theorem.

\begin{theorem}\label{A}
\begin{itemize}
\item [(a)] Suppose that $\mathbf{k}\in  \mathcal{PS} \cap \mathcal{S}.$ 
\begin{itemize}
\item [(a-1)] The open segment
\[
R=\left\{\left( \frac {k_3}{k_4}z, \frac {k_4-(k_4 +k_3)z}{k_4+k_1}, z\right): 0<z<\frac {k_4}{k_3+k_4}\right\}
\]
is contained in the interior of $\mathcal{T}$ and every point in $R$ is a singular point.
\item[(a-2)] Let $\mathbf{p}$ be a point contained in the interior of $\mathcal{T}$ but not in $R.$ Then the orbit $\gamma_{\mathbf{p}}$ through the point $\mathbf{p}$ is a periodic orbit.
\item[(a-3)] Each of the two limit sets of every orbit in $\mathcal{Y}\cup \Sigma$ is a singular point contained in the edge $\mathcal{R}_{py}.$ Moreover, given two orbits $\gamma_1 \subset \mathcal{Y}$ and $\gamma_2 \subset \Sigma$ such that $\omega(\gamma_1)=\alpha(\gamma_2),$ then $\omega(\gamma_2) = \alpha(\gamma_1).$
\item[(a-4)] Each of the two limit sets of every orbit in $\mathcal{X}\cup \mathcal{Z}$ is a singular point contained in the edge $\mathcal{R}_{xz}.$ Moreover, given two orbits $\gamma_1 \subset \mathcal{X}$ and $\gamma_2 \subset \mathcal{Z}$ such that $\omega(\gamma_1)=\alpha(\gamma_2),$ then $\omega(\gamma_2) = \alpha(\gamma_1).$
\end{itemize}
\item [(b)] Suppose that $\mathbf{k} \not \in  \mathcal{PS} \cap \mathcal{S}$ and $\mathbf{k} \neq \mathbf{0}.$ The limit sets of every orbit in $\mathcal{T}$ are contained in the boundary of $\mathcal{T}$ and these limit sets are non--periodic orbits. 
\end{itemize}
\end{theorem}

\begin{figure}
\begin{center}
\includegraphics[width=10cm,height=5cm]{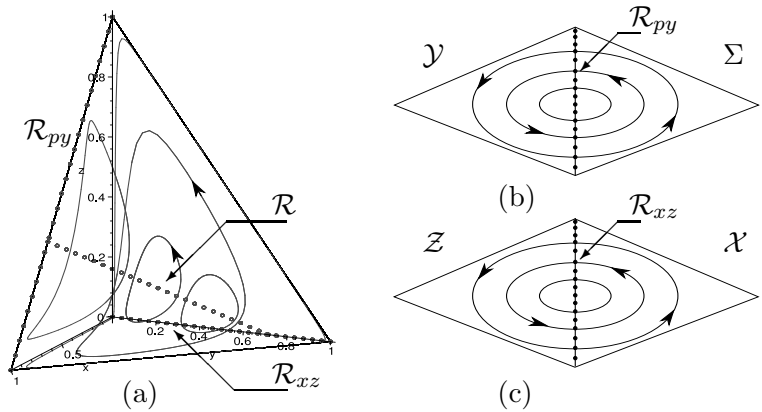}
\end{center}
\begin{picture}(0,0)(0.7,-0.1)
\put(4.6,0.5){(a)}
\put(3.3,4){$\mathcal{R}_{py}$}
\put(6.6,0.75){$\mathcal{R}_{xz}$}
\put(6.8,0.65){\line(-1,0){0.7} \vector(-1,1){0.8}}
\put(11.35,5.5){$\mathcal{R}_{py}$}
\put(11.5,5.4){\line(-1,0){0.3} \vector(-1,-1){0.5}}
\put(11.35,3){$\mathcal{R}_{xz}$}
\put(11.5,2.9){\line(-1,0){0.3} \vector(-1,-1){0.5}}
\put(9.6,3.1){(b)}
\put(8.6,5){$\mathcal{Y}$}
\put(12.6,5){$\Sigma$}
\put(9.6,0.5){(c)}
\put(8.6,2.5){$\mathcal{Z}$}
\put(12.6,2.5){$\mathcal{X}$}
\put(6.6,3){$\mathcal{R}$}
\put(6.8,2.9){\line(-1,0){0.8} \vector(-1,-1){0.8}}
\end{picture}
\caption{Behaviour of the flow of system (\ref{3DLV}) for $\mathbf{k}\in \mathcal{PS} \cap \mathcal{S}:$ (a) in the interior of $\mathcal{T},$ which contains periodic orbits and the segment $\mathcal{R}$ formed by singular points; (b) at the faces $\mathcal{Y}\cup \Sigma$ which exhibit an edge $\mathcal{R}_{py}$ formed by singular points, and heteroclinic orbits giving rise to heteroclinic loops; (c) at the faces $\mathcal{X}\cup \mathcal{Z}$ which exhibit an edge $\mathcal{R}_{xz}$ formed by singular points, and heteroclinic orbits giving rise to heteroclinic loops. \label{tetra2}}\vspace{-0.5cm}
\end{figure}

From Theorem \ref{A}(a) it follows that when $\mathbf{k}\in \mathcal{PS} \cap \mathcal{S}$ the behaviour of the trajectories in the whole region $\mathcal{T}$  is illustrated in Figure \ref{tetra2}.

Theorem \ref{A}(b) is partially a corollary of a general theorem for Lotka--Volterra systems stating that no limit points are in the interior of the non--negative orthant if there is no singular points in the interior. See Theorem 5.2.1 (p. 43) in \cite{HS02}.

On the other hand Theorem \ref{A} completely characterizes the region in the parameter space where the corresponding system (\ref{3DLV}) exhibits self sustained oscillations. Thus the necessary conditions $\mathbf{k}\in\mathcal{S}$ for the existence of such behaviour, given by Di Cera et al. \cite{DPW88}, are here completed with the necessary and sufficient condition  $\mathbf{k}\in\mathcal{PS}\cap \mathcal{S}.$  Furthermore this oscillating behaviour in the interior of $\mathcal{T}$ extends to a heteroclinic behaviour at the boundary. Therefore the period function defined in the interior of $\mathcal{T}$ is a non--constant function; it grows when approaching the boundary. As a final remark, we note that oscillations in system (\ref{3DLV}) take place in a parameter set of measure zero, and they are only ``conservative'' ones; i.e. are not isolated in the set of all periodic orbits.

In dimension greater than two continuous dynamical systems may present chaotic motion, in the sense that the distance of points on trajectories starting close together increases on an exponential rate. This is not our case. The dynamic behaviour of family (\ref{3DLV}) is very simple and non-strange attractors appear. In fact, as shown in Theorem \ref{A}(b) in absence of periodic orbits every orbit goes from one side of the boundary of $\mathcal{T}$ to another. Nevertheless, we can remark certain singular situations related to the form and location of the limit sets. One of these limit set configurations is described in the next result. Before stating it we consider the following singular points in the edges $\mathcal{R}_{py}$ and $\mathcal{R}_{xz},$ respectively 
\begin{equation}\label{ptoslimite}
\begin{array}{ll}
\mathbf{p}_{py}=\left( \frac {k_2}{k_1+k_2},0,\frac {k_1}{k_1+k_2}\right),&
\mathbf{q}_{py}=\left( \frac {k_3}{k_3+k_4},0,\frac {k_4}{k_3+k_4}\right),\\
& \\
\mathbf{p}_{xz}=\left( 0, \frac {k_4}{k_1+k_4},0 \right),&
\mathbf{q}_{xz}=\left( 0, \frac {k_3}{k_3+k_2},0 \right).
\end{array}
\end{equation}

When $\mathbf{k}$ is in the manifold $\mathcal{PS} \cap \mathcal{S},$ the points $\mathbf{p}_{py}$ and $\mathbf{q}_{py}$ are equal and they coincide  with one of the endpoints of the segment $R$ defined in Theorem \ref{A}(a-1). Similarly, the points  $\mathbf{p}_{xz}$ and $\mathbf{q}_{xz}$ are also equal and they coincide with the other endpoint of $R.$ On the other hand, when $\mathbf{k} \in  \mathcal{PS} \setminus \mathcal{S},$ we define the following segments contained in the edges $\mathcal{R}_{py}$ and $\mathcal{R}_{xz},$ respectively
\begin{eqnarray*}\label{seglimite}
{s}_{py}&=\left\{ \mathbf{p}_{py}+r(\mathbf{q}_{py}-\mathbf{p}_{py}): r \in [0,1] \right\},\\
{s}_{xz}&=\left\{ \mathbf{p}_{xz}+r(\mathbf{q}_{xz}-\mathbf{p}_{xz}): r \in [0,1] \right\}.
\end{eqnarray*}

To clarify the exposition of the next result we introduce the subsets $\mathcal{PS}_+=\{\mathbf{k}\in \mathcal{P}: k_i>0\}$ and $\mathcal{PS}_-=\{ \mathbf{k} \in \mathcal{P}: k_i<0\}$ which form a partition of $\mathcal{PS}.$

\begin{theorem}\label{B} Suppose that $\mathbf{k} \in \mathcal{PS} \setminus \mathcal{S}.$ 
\begin{itemize}
\item [(a)] Each of the two limit sets of every orbit in the interior of $\mathcal{T}$ is formed by a singular point contained in the segments $s_{py}$ and $s_{xz}.$ In particular, given a point $\mathbf{p}$ in the interior of $\mathcal{T},$ if $\mathbf{k} \in \mathcal{PS}_+ \cap \mathcal{S}^+ $ or $\mathbf{k} \in \mathcal{PS}_-\cap \mathcal{S}^-,$ then $\alpha(\gamma_{\mathbf{p}}) \in s_{xz}$ and $\omega(\gamma_{\mathbf{p}}) \in s_{py};$ and if $\mathbf{k} \in \mathcal{PS}_+ \cap \mathcal{S}^-$ or $\mathbf{k} \in \mathcal{PS}_- \cap \mathcal{S}^+$ then $\alpha(\gamma_{\mathbf{p}}) \in s_{py}$ and $\omega(\gamma_{\mathbf{p}}) \in s_{xz}.$
\item [(b)] Each of the two limit sets of every orbit in $\mathcal{Y}\cup \Sigma$ is a singular point contained in the edge $\mathcal{R}_{py}.$ Moreover, given two orbits $\gamma_1 \subset \mathcal{Y}$ and $\gamma_2 \subset \Sigma$ such that $\omega(\gamma_1)=\alpha(\gamma_2),$ then $\omega(\gamma_2) \neq \alpha(\gamma_1).$
\item [(c)] Each of the two limit sets of every orbit in $\mathcal{X}\cup \mathcal{Z}$ is a singular point contained in the edge $\mathcal{R}_{xz}.$ Moreover, given two orbits $\gamma_1 \subset \mathcal{X}$ and $\gamma_2 \subset \mathcal{Z}$ such that $\omega(\gamma_1)=\alpha(\gamma_2),$ then $\omega(\gamma_2) \neq \alpha(\gamma_1).$
\end{itemize}
\end{theorem}

\begin{figure}
\begin{center}
\includegraphics[width=10cm,height=5cm]{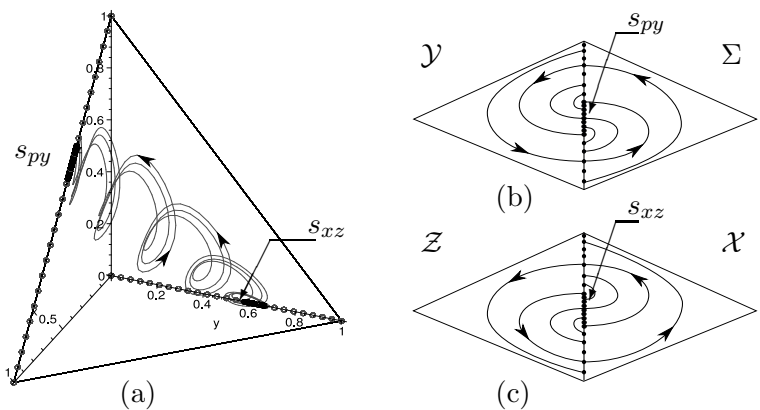}
\end{center}
\begin{picture}(0,0)(0.7,-0.1)
\put(4.6,0.5){(a)}
\put(3.2,3.75){${s}_{py}$}
\put(7.1,2.75){${s}_{xz}$}
\put(7.2,2.65){\line(-1,0){0.6} \vector(-1,-2){0.4}}
\put(11.35,5.5){$s_{py}$}
\put(11.5,5.4){\line(-1,0){0.3} \vector(-1,-3){0.35}}
\put(11.35,3){$s_{xz}$}
\put(11.5,2.9){\line(-1,0){0.3} \vector(-1,-3){0.35}}
\put(9.6,3.1){(b)}
\put(8.6,5){$\mathcal{Y}$}
\put(12.6,5){$\Sigma$}
\put(9.6,0.5){(c)}
\put(8.6,2.5){$\mathcal{Z}$}
\put(12.6,2.5){$\mathcal{X}$}
\end{picture}
\caption{\label{tetra3}Behaviour of the flow of system (\ref{3DLV}) for $\mathbf{k}\in \mathcal{PS}_+ \cap \mathcal{S}^+ $ or  for $\mathbf{k}\in \mathcal{PS}_- \cap \mathcal{S}^- :$ (a) the limit set of every orbit in the interior of $\mathcal{T}$ is a point contained in the segments $s_{xz}$ and $s_{py},$ respectively; (b) at the faces $\mathcal{Y}\cup \Sigma$ where the orbits are heteroclinic orbits connecting singular points located at the edge $\mathcal{R}_{py},$ these heteroclinic orbits do not form heteroclinic loops; (c) at the faces $\mathcal{X}\cup \mathcal{Z}$ where the orbits are heteroclinic orbits connecting singular points located at the edge $\mathcal{R}_{xz},$ these heteroclinic orbits do not form heteroclinic loops.}\vspace{-1.0cm}
\end{figure}

Therefore when $\mathbf{k} \in \mathcal{PS}_+ \cap \mathcal{S}^+ $ or $\mathbf{k} \in \mathcal{PS}_-\cap \mathcal{S}^-,$ the behaviour of the trajectories in the whole region $\mathcal{T}$ is represented in Figure \ref{tetra3}. Moreover since a change in the sign of the parameter $\mathbf{k}$ has the same effect than a change in the sign of the time variable, see system (\ref{3DLV}), the behaviour of the trajectories when $\mathbf{k} \in \mathcal{PS}_- \cap \mathcal{S}^+ $ or $\mathbf{k} \in \mathcal{PS}_+\cap \mathcal{S}^-,$ follows by changing the direction of the flow in Figure \ref{tetra3}.

From Theorem \ref{B} we conclude that the bifurcation taking place at the manifold $\mathcal{S}$ is not only characterized by the behaviour of the flow in the interior of $\mathcal{T}.$ In addition, it must be described by taking into account the changes of the limit sets  $s_{xz}$ and $s_{py}$ at the boundary of $\mathcal{T}$. 
Hence when $\mathbf{k} \in \mathcal{PS}_+ \cap \mathcal{S}^+$ the orbits in the faces   $\mathcal{X}\cup\mathcal{Z}$ are organized in spirals around the segment $s_{xz}$  moving away from it; and the orbits in the faces   $\mathcal{Y}\cup \Sigma$ are organized in spirals around the segment $s_{py}$  approaching it. When $\mathbf{k} \in \mathcal{PS} \cap \mathcal{S},$ the segment $s_{xz}$ reduces to the singular point $\mathbf{p}_{xz}$ and the segment $s_{py}$ reduce to the singular point $\mathbf{p}_{py};$ furthermore the flow in the faces $\mathcal{X}\cup\mathcal{Z}$ and $\mathcal{Y} \cup \Sigma$ describes heteroclinic orbits around them. Finally, when $\mathbf{k} \in \mathcal{PS}_+ \cap \mathcal{S}^-$ the orbits in $\mathcal{X}\cup\mathcal{Z}$ are organized in spiral around the segment $s_{xz}$  approaching it; and the orbits in the faces $\mathcal{Y} \cup \Sigma$ are organized in spirals around the segment $s_{py}$ moving away from it. 
From this we denote the bifurcation taking place at the manifold $\mathcal{PS} \cap \mathcal{S}$ by a focus--center--focus bifurcation. The bifurcation set of system (\ref{3DLV}) is drawn in Figure \ref{fig:bifset}.

The paper is organized as follows. In Section 2 we analyze the existence and the local behaviour of the singular points both in the interior and in the boundary of $\mathcal{T}.$ In Section 3 we deal with the first integrals of the flow and we characterize the integrability of the flow. Using these first integrals, in Section 4 we analyze the flow at the boundary of $\mathcal{T}.$ In Section 5 and by using again the first integrals we analyze the flow in the interior of $\mathcal{T}$ and we prove the main results of the paper.

\begin{figure}
\begin{center}
\includegraphics{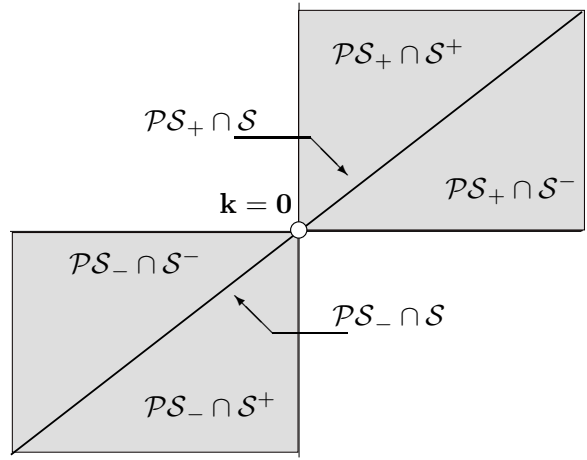}
\end{center}
\begin{picture}(0,0)(0.75,-0.1)
\put(8.6,6){$\mathcal{PS}_+ \cap \mathcal{S}^+$}
\put(10.1,4.2){$\mathcal{PS}_+ \cap \mathcal{S}^-$}
\put(5.1,3.2){$\mathcal{PS}_- \cap \mathcal{S}^-$}
\put(6.1,1.3){$\mathcal{PS}_- \cap \mathcal{S}^+$}
\put(6.1,5.1){$\mathcal{PS}_+ \cap \mathcal{S}$}
\put(7.3,5){\line(1,0){1} \vector(1,-1){0.5}}
\put(8.6,2.5){$\mathcal{PS}_- \cap \mathcal{S}$}
\put(8.8,2.4){\line(-1,0){1} \vector(-1,1){0.5}}
\put(7.1,4){$\mathbf{k} = \mathbf{0}$}
\end{picture}
\caption{\label{fig:bifset}Representation of the bifurcation set in a two dimensional parameter space.}\vspace{-0.75cm}
\end{figure}

\section{Singular points}

In the following proposition we summarise the results about the existence, location and stability of the singular points of system (\ref{3DLV}).

\begin{proposition} \label{ptsing}The half straight lines $\mathcal{R}_{py}$ and $\mathcal{R}_{xz}$ are formed by singular points.
\begin{itemize}
\item [(a)] If $\mathbf{k} \in \mathcal{NZ}$ there are no other singular points in the boundary of the simplex.
\begin{itemize}
\item [(a-1)] Suppose that $\mathbf{k} \in \mathcal{PS} \cap \mathcal{S}.$ The open segment
\[
R=\left\{\left( \frac {k_3}{k_4}z, \frac {k_4-(k_4 +k_3)z}{k_4+k_1}, z\right): 0<z<\frac {k_4}{k_3+k_4}\right\}
\]
is formed by all the singular points in the interior of the region $\mathcal{T}.$ Moreover the Jacobian matrix of the vector field evaluated at each of these points has one real eigenvalue equal to zero and two purely imaginary eigenvalues.
\item [(a-2)] Suppose that $\mathbf{k} \in \mathcal{NZ} \setminus \left\{\mathcal{PS} \cap \mathcal{S} \right\}.$ There are no singular points in the interior of region $\mathcal{T}.$
\end{itemize}
\item[(b)] Suppose that $\mathbf{k}\not \in \mathcal{NZ}$ and $\mathbf{k} \neq \mathbf{0},$ in this case there are no singular points in the interior of  $\mathcal{T}.$ In fact the singular points are on the boundary of $\mathcal{T}$ and they complete either edges or whole faces. 
\end{itemize}
\end{proposition}

\begin{proof}
Straightforward computations show that the half straight lines $\mathcal{R}_{py}$ and $\mathcal{R}_{xz}$ are formed by singular points.

(a-1) Suppose now that $\mathbf{k} \in \mathcal{NZ}.$ Hence none of the components of the parameter $\mathbf{k}$ is zero. In this case the singular points are given by the solutions to the following systems
\begin{equation*}
\left.
\begin{array}{r}
x=0\\
y z =0\\
z(-k_2 y + k_3 (1-y-z))=0
\end{array}
\right\}\hspace{0.25cm}
\left.
\begin{array}{r}
-x(1-x-z)=0\\
y=0\\
 z(1-x-z)=0
\end{array}
\right\} 
\end{equation*}
\begin{equation*}
\left.
\begin{array}{r}
x(k_1 y-k_4(1-x-y))=0\\
- yx =0\\
 z =0 
\end{array}\right\} \hspace{0.25cm}
\left.
\begin{array}{r}
k_4x +(k_1+k_4)y + k_4 z= k_4\\
(k_1 +k_4) y +(k_3+k_4) z =k_4\\
(k_2 k_4-k_1 k_3) (y + z) =k_2 k_4-k_1 k_3
\end{array}
\right\}
\end{equation*}
where in the last one we impose $x y z \neq 0$ to avoid repetitions.
From the three first systems it is easy to conclude that there are no other singular points than those in the half straight lines $\mathcal{R}_{py}$ and $\mathcal{R}_{xz}.$ With respect to the last one we distinguish two situations. 

First let us suppose that $\mathbf{k} \not \in S,$ that is $k_2 k_4-k_1 k_3 \not = 0.$ From the third equation it follows that $y+z=1,$ and therefore $x=0.$ Since $k_1 k_4 \neq 0$ from the first equation we conclude that $y=0$ and $z=1.$ Hence the singular point is one of the endpoints of the edge $\mathcal{R}_{xz};$ i.e. it does not belong to the interior of $\mathcal{T}.$

Suppose now that $\mathbf{k} \in S,$ that is $k_2 k_4-k_1 k_3 = 0.$ Thus the linear system is equivalent to the following one
\[
\left.
\begin{array}{rl}
k_4x +(k_1+k_4)y + k_4 z&= k_4\\
(k_1 +k_4) y +(k_3+k_4) z &=k_4.
\end{array} \right\}
\]
If $k_1 +k_4=0,$ then from the first equation we obtain $x+z=1.$ Therefore $y=0$ and the singular point belongs to $\mathcal{R}_{py}.$ On the contrary, if $k_1 +k_4 \not =0,$ then there exists a straight line of singular points parametrically defined by $x=z k_3/k_4 $ and $y=(k_4-(k_3+k_4) z)/(k_1+k_4).$ Since the singular points in the interior of $\mathcal{T}$ must satisfy that $x > 0, y > 0, z > 0$ and $x+y+z < 1,$ then there exist singular points in the interior of $\mathcal{T}$ if and only if 
\[
\frac {k_3}{k_4} > 0,\quad \frac {k_3+k_4}{k_1+k_4}z < \frac {k_4}{k_1+k_4},\quad \frac {k_1}{k_4} \left(\frac {k_3+k_4}{k_1+k_4} z \right) < \frac {k_1}{k_1+k_4},\quad z > 0.
\]
It is easy to check that the previous inequalities are equivalent to 
\[
\frac {k_3}{k_4} > 0,\quad \frac {k_3+k_4}{k_1+k_4}z < \frac {k_4}{k_1+k_4},\quad \frac {k_1}{k_4} > 0,\quad z>0.
\]
Since $\mathbf{k}\in \mathcal{S}$ we have $k_1/k_4=k_2/k_3.$ Therefore we conclude that there exist singular points in the interior of $\mathcal{T}$ if and only if all the components of $\mathbf{k}$ have the same sign; that is $\mathbf{k} \in \mathcal{PS}.$ In such case these singular points are given by  
\begin{equation}\label{intsingpoints}
x=\frac {k_3}{k_4}z,\quad y=\frac{k_4-(k_3+k_4) z}{k_1+k_4},\quad 0 < z < \frac {k_4}{k_1+k_4},
\end{equation}
which proves statement (a-1).

(a-2) The Jacobian matrix of the vector field defined by the differential equation (\ref{3DLV}) and evaluated at the singular points (\ref{intsingpoints}) is given by
\[
\left(
\begin{array}{ccc}
k_3 z & (k_2+k_3) z & k_3 z \\
- k_1 y & 0 & k_2 y \\
-k_3 z & -(k_2+k_3) z & -k_3 z \\
\end{array}
\right)
\]
The characteristic polynomial is equal to $\lambda(\lambda^2+b)=0,$ where $b=zy(k_1+k_2)(k_2+k_3).$ Since $\mathbf{k} \in \mathcal{PS}$ the coefficient $b$ is positive. Then we get one zero eigenvalue and a pair of complex conjugated eigenvalues with zero real part.

(b) If $\mathbf{k} \not \in \mathcal{NZ}$ and $\mathbf{k} \neq \mathbf{0},$ then at least one of the coordinates of $\mathbf{k}$ is equal to zero and at least one is different from zero. Without loss of generality we suppose that $k_1=0$ and $k_2 \neq 0.$ From the second equation in (\ref{3DLV}) it follows that the coordinates of the singular points satisfy $yz=0.$ Therefore the singular points are contained in the boundary of $\mathcal{T}.$ Moreover from the remainder equations in (\ref{3DLV}) it follows that $-k_4x(1-x-y-z) =0$ and $k_3z(1-x-y-z) =0.$ We conclude that, depending on whether the parameters $k_3$ and $k_4$ are zero or not, singular points complete either whole faces or edges, respectively. 
\end{proof}

In the next result we deal with the singular points located at the edges $\mathcal{R}_{py}$ and $\mathcal{R}_{xz}$ which are not on the segments $s_{py}$ and $s_{xz},$ respectively. Note that these points are not hyperbolic singular points, so that we can not apply Hartman--Grobman Theorem to describe the behaviour of the flow in a neighbourhood of them.

\begin{proposition}\label{prop:ptssing}
If $\mathbf{k} \in \mathcal{PS} \setminus \mathcal{S},$ then no singular point in  $\mathcal{R}_{py} \setminus s_{py}$ and $\mathcal{R}_{xz} \setminus s_{xz}$ is the limit set of an orbit in the interior of the region $\mathcal{T}.$ 
\end{proposition}

\begin{proof}
Let $\mathbf{p}$ be a point in the set $\mathcal{R}_{py} \setminus s_{py},$ that is $\mathbf{p}=(x_0,0,1-x_0)$ where  either 
\begin{equation}\label{eq:xo}
x_0 > \max\left\{ \frac {k_2}{k_1+k_2}, \frac {k_3}{k_3 +k_4} \right\}
\quad\mbox{or}\quad
x_0 < \min \left\{ \frac {k_2}{k_1+k_2}, \frac {k_3}{k_3 +k_4}\right\},
\end{equation}
see expression (\ref{ptoslimite}). If we consider a point $\mathbf{p}$ in the set $\mathcal{R}_{xz} \setminus s_{xz},$ the following arguments can be applied in a similar way.

Through the change of variables $\bar{x}=x-x_0,$ $\bar{y}=y$ and $\bar{z}=z-1+x_0,$ system (\ref{3DLV}) can be written as system $\dot{\bar{\mathbf{x}}}=A\bar{\mathbf{x}}+\mathbf{Q}(\bar{\mathbf{x}})$ where $\bar{\mathbf{x}}=(\bar{x},\bar{y},\bar{z})^T,$   
\begin{equation*}
 A=\left(
\begin{array}{ccc}
k_4 x_0 & (k_1+k_4) x_0  & k_4 x_0 \\
0   &  k_2-(k_1+k_2) x_0 & 0 \\
k_3 (x_0 -1) & (k_2+k_3)(x_0-1) & k_3(x_0-1)  
\end{array}
\right)
\end{equation*}
and
\begin{equation*}
\mathbf{Q}(\bar{\mathbf{x}})=\left(
\begin{array}{c}
\bar{x}(k_4 \bar{x}+(k_1+k_4) \bar{y}+k_4 \bar{z}) \\
\bar{y}(k_2 \bar{z}-k_1 \bar{x}) \\
\bar{z}(-k_3 \bar{x}-(k_2+k_3) \bar{y} -k_3 \bar{z})
\end{array}
\right).
\end{equation*}
The eigenvalues of the matrix $A$ are $\lambda_1=0, \lambda_2=(k_3+k_4)x_0-k_3$ and $\lambda_3=k_2-x_0(k_1+k_2).$ From (\ref{eq:xo}) it is easy to conclude that $\lambda_2 \lambda_3 <0.$ Therefore there exists a regular matrix $P$ such that $PAP^{-1}=diag\{0,\lambda_2,\lambda_3\}.$ 

Going through the change of coordinates $\mathbf{x}_{p}=P\bar{\mathbf{x}}$ the system can be rewritten as 
\begin{equation}\label{sys:xp}
\left\{
\begin{array}{l}
\dot{x}_{p}=\frac {k_2 z_{p}-k_3 y_{p} }{k_1k_4x_0} \left(k_4k_1 x_{p}+k_1(x_0-1)(k_3+k_4)y_p+k_4(x_0-1)(k_2+k_4)z_p\right) \\ \\
\dot{y}_p=\frac {y_p}{k_1k_4x_0}
\left( 
(\lambda_2-(k_3+k_4)x_p)k_1k_4x_0+k_1(k_3^2(1-x_0)+k_4^2x_0)y_p \right.\\ \\
\left.
\hspace{2cm}+k_4(k_1 k_4 x_0 +k_2k_3(1-x_0))z_p
\right)\\ \\
\dot{z}_p=\frac {z_p}{k_1k_4x_0}
\left(
(\lambda_3+(k_2+k_1)x_p)k_1k_4x_0-k_1(k_1k_4x_0+k_2k_3(1-x_0))y_p \right. \\ \\
\left.
\hspace{2cm}- k_2(k_2^2(1-x_0)+k_1^2x_0)z_p
\right)
\end{array}
\right.
\end{equation}
System (\ref{sys:xp}) has two invariant planes $\{y_p=0\}$ and $\{z_p=0\}$ intersecting at a straight line formed by singular points, which corresponds to the segment $\mathcal{R}_{py}.$ The direction of the vector field in a sufficiently small neighbourhood of the origin satisfies that
\begin{equation*}
\begin{array}{l}
{\rm sign}(\dot{y}_p)={\rm sign}(y_p) {\rm sign}(\lambda_2)\\ \\
{\rm sign}(\dot{z}_p)={\rm sign}(z_p) {\rm sign}(\lambda_3).
\end{array}
\end{equation*}
We conclude that the origin is neither the $\alpha$--limit set nor the $\omega$--limit set of any orbit in the interior of the regions $\{y_p>0,z_p>0\}, \{y_p>0, z_p<0\},\{y_p<0,z_p>0\}$ and $\{y_p<0,z_p<0\}.$ From this we conclude the proposition. 
\end{proof}

\section{Invariant algebraic surfaces and first integrals}

In 1878 Darboux showed how to construct first integrals of a planar polynomial vector field possessing sufficient invariant algebraic curves. Recent works improved the Darboux's exposition taking into account other dynamical objects like exponential factors and independent singular points,  see \cite{C94}, \cite{CL99} and \cite{CLS97} for more details. The extension of the Darboux theory to $n$--dimensional systems of polynomial differential equations can be found in the work by Llibre and Rodr\'{\i}guez \cite{LR00}. A brief introduction to the three dimensional case can be found in \cite{CL00}

Following \cite{CL00} a first integral of system (\ref{3DLV}) is a real function $F$ non--constant over the region $\mathcal{T}$ and such that the level surfaces $\mathcal{F}_C=\{(x,y,z)\in \mathcal{T}: F(x,y,z)=C\}$ are invariants by the flow; that is
\[
XF=\frac {\partial F}{\partial x} \dot{x} + \frac {\partial F}{\partial y}\dot{y} + \frac {\partial F}{\partial z}\dot{z}=0
\]
where $X=\dot{x}\frac {\partial}{\partial x} +\dot{y}\frac{\partial}{\partial y}+\dot{z}\frac{\partial}{\partial z}$ is the vector field associated to the system of differential equations. Thus the existence of a first integral allows the reduction of the dimension of the problem by one. Moreover, the existence of two independent first integrals allows the integrability of the flow. 

Let $f\in \mathbb{R}[x,y,z]$ be a polynomial function. The algebraic surface $f=0$ is called an invariant algebraic surface of the system (\ref{3DLV}) if there exists a polynomial $K\in \mathbb{R}[x,y,z]$ such that $Xf=Kf.$ The polynomial $K$ is called the cofactor of $f.$ The following result is a corollary of Theorem 2 in \cite{CL00}.

\begin{theorem}\label{jaume}
Suppose that the polynomial vector field (\ref{3DLV}) admits p invariant algebraic surfaces $f_i=0$ with cofactors $K_i$ for $i=1,2,\ldots,p.$ If there exist $\lambda_i \in \mathbb{R}$ not all zero such that $\sum_{i=1}^p \lambda_i K_i=0,$ then the function $f_1^{\lambda_1}f_2^{\lambda_2}\ldots f_p^{\lambda_p}$ is a first integral of the vector field (\ref{3DLV}).
\end{theorem}

Now we deal with the existence of Darboux type first integrals of system (\ref{3DLV}). Consider the algebraic surfaces $f_1(x,y,z)=x,$ $f_2(x,y,z)=y,$ $f_3(x,y,z)=z$ and $f_4(x,y,z)=x+y+z-1.$ It is easy to check that $X f_i =f_i K_i,$ with $i=1,2,3,4,$ where $K_1(x,y,z)=k_4 x+ (k_1+k_4)y+k_4 z-k_4,$ $K_2(x,y,z)=-k_1 x + k_2 z,$ $K_3(x,y,z)=-k_3 x-(k_2+k_3)y-k_3 z+k_3$ and $K_4(x,y,z)=k_4 x -k_3 z.$
Therefore $f_i=0$ is an invariant surface with cofactor $K_i,$ with $i=1,2,3,4.$

From Theorem \ref{jaume}, if there exist $\lambda_i$ not all zero and such that $\sum_{i=1}^4 \lambda_i K_i=0,$ then $F=f_1^{\lambda_1}f_2^{\lambda_2}f_3^{\lambda_3}f_4^{\lambda_4}$
is a first integral of system (\ref{3DLV}).
Since
\begin{eqnarray*}
\sum_{i=1}^4 \lambda_i K_i&=&(\lambda_4 k_4-\lambda_2 k_1) x+(\lambda_1 k_1- \lambda_3 k_2) y \\ &+&(\lambda_2 k_2-\lambda_4 k_3) z+ (\lambda_3 k_3 -\lambda_1 k_4) (1-x-y-z)
\end{eqnarray*}
the existence of such $\lambda_i$ is equivalent to the existence of non--trivial solutions of the homogeneous linear systems
\begin{equation}\label{firstintsys}
\left(
\begin{array}{cc}
k_1 & -k_2 \\
-k_4 & k_3
\end{array}
\right)
\left(
\begin{array}{c}
\lambda_1\\
\lambda_3
\end{array}
\right)=
\left(
\begin{array}{c}
0\\
0 \\
\end{array}
\right)
\quad\mbox{and}\quad
\left(
\begin{array}{cc}
k_2 & -k_3 \\
-k_1 & k_4
\end{array}
\right)
\left(
\begin{array}{c}
\lambda_2\\
\lambda_4
\end{array}
\right)=
\left(
\begin{array}{c}
0\\
0 \\
\end{array}
\right).
\end{equation}
Note that the determinant of both previous systems is equal to $k_1 k_3 - k_2 k_4.$ Therefore when $\mathbf{k}$ belongs to the set $\mathcal{S}$ there exist Darboux type first integrals of system (\ref{3DLV}).

Under the assumption $\mathbf{k} \in \mathcal{S}$ the linear system (\ref{firstintsys}) has the following non--trivial solutions $(\lambda_1,\lambda_2,\lambda_3,\lambda_4) \to (k_2,0,k_1,0), (0,k_3,0,k_2), (k_3,0,k_4,0)$ and $(0,k_4,0,k_1).$ Therefore the functions  $H(x,y,z)=x^{k_2}z^{k_1},$ $V(x,y,z)=y^{k_3}(1-x-y-z)^{k_2},$ $\widetilde{H}(x,y,z)=x^{k_3}z^{k_4}$ and  $\widetilde{V}=y^{k_4}(1-x-y-z)^{k_1}$ are first integrals. In fact

\begin{equation}\label{xporhyv}
\begin{array}{l}
XH=x^{k_2}z^{k_1}(1-x-y-z)(k_1 k_3 - k_2 k_4 ) \\
XV=y^{k_3}(1-x-y-z)^{k_2}x(k_2 k_4 -k_1 k_3),\\ 
X\widetilde{H}=x^{k_3}z^{k_4}(k_1 k_3 - k_2 k_4)y \\
X\widetilde{V}=y^{k_4}(1-x-y-z)^{k_1}(k_2k_4 - k_1k_3)z
\end{array}
\end{equation}
which vanish in the whole region $\mathcal{T}$ only when $\mathbf{k} \in \mathcal{S}.$  

\begin{proposition}\label{firstint} Consider the functions $H(x,y,z)=x^{k_2}z^{k_1},$ $\widetilde{H}(x,y,z)=x^{k_3}z^{k_4},$  $V(x,y,z)=y^{k_3}(1-x-y-z)^{k_2}$ and $\widetilde{V}(x,y,z)=y^{k_4}(1-x-y-z)^{k_1}.$ 
\begin{itemize}
\item [(a)] If $\mathbf{k} \in \mathcal{S} \cap \mathcal{NZ},$ then $H, V, \widetilde{H}$ and $\widetilde{V}$ are first integrals which satisfy that $\widetilde{H}^{k_1}=H^{k_4}$ and $\widetilde{V}^{k_3}=V^{k_4}.$ Moreover $H$ and $V$ are independent. 
\item [(b)] If $\mathbf{k} \in \mathcal{S} \setminus \mathcal{NZ},$ then two of the previous functions are first integrals and they are independent.
\item [(c)] If $\mathbf{k} \not \in \mathcal{S},$ then none of the previous functions is a first integral in $\mathcal{T}.$
\end{itemize}
\end{proposition}

\begin{proof}
(a) Consider that $\mathbf{k}\in \mathcal{S} \cap \mathcal{NZ}.$ Since every coordinate of $\mathbf{k}$ is different from zero it follows that $H,V, \widetilde{H}$ and $\widetilde{V}$ are not constant in $\mathcal{T}.$ Therefore all of these functions are first integrals. It is easy to check that  $\widetilde{H}^{k_1}=H^{k_4}$ and $\widetilde{V}^{k_3}=V^{k_4}.$ Moreover since
$\nabla H(x,y,z)=x^{(k_2 -1)}z^{(k_1-1)}\left(k_2 z,0,k_1 x \right)$ and $\nabla V (x,y,z) = y^{(k_3 -1)}(1-x-y-z)^{(k_2-1)}\left( -k_2 y, k_3(1-x-y-z)-k_2y,-k_2 y \right),$ both integrals are dependent only on points satisfying $k_3(1-x-y-z)=k_2 y$ and $k_2 z =k_1 x.$ Taking into account that $k_2 \neq 0$ it follows that this set has zero Lebesgue measure. Then $H$ and $V$ are two independent first integrals.

(b) Consider now that $\mathbf{k}\in \mathcal{S} \setminus \mathcal{NZ}.$ Hence $\mathbf{k}$ has one coordinate which is different from zero. Without loss of generality we assume that $k_1 \neq 0,$ the remainder cases follows in a similar way. It is easy to check that $H$ and $\widetilde{V}$ are not constant in $\mathcal{T},$ and therefore they are first integrals.  Since $\nabla H(x,y,z)=x^{(k_2 -1)}z^{(k_1-1)}\left(k_2 z,0,k_1 x \right)$ and $\nabla \widetilde{V} (x,y,z) = y^{(k_4 -1)}(1-x-y-z)^{(k_1-1)}\left( -k_1 y, k_4(1-x-y-z)-k_1y,-k_1 y \right),$ both integrals may be dependent only on points satisfying $k_4(1-x-y-z)=k_1 y$ and $k_2 z =k_1 x.$ Therefore $H$ and $\widetilde{V}$ are independent. 

(c) The statement follows straightforward form expression (\ref{xporhyv}).
\end{proof}

\section{Behaviour at the boundary}

As we have proved in Proposition \ref{firstint} some of the functions  $H, \widetilde{H}, V$ and $\widetilde{V}$ are first integrals over the whole region $\mathcal{T}$ only when $\mathbf{k}\in \mathcal{S}.$  Nevertheless the restriction of these functions to a particular face of $\mathcal{T}$ results in a first integral even when $\mathbf{k} \not \in \mathcal{S}.$ In fact, denoting by $\widetilde{H}|_{\mathcal{Y}}$ the restriction of the function $\widetilde{H}$ to the face $\mathcal{Y},$ from  expression (\ref{xporhyv}) it follows that $X \widetilde{H}|_{\mathcal{Y}}=0.$  Therefore the level curves $\widetilde{H}|_{\mathcal{Y}}=C^{k_4}$ are invariant by the flow. Under the assumption $k_3 k_4 >0,$ these level curves define a foliation of $\mathcal{Y}$ whose leaves are given by the arcs of hyperbolas  $\left\{ z=C x^{- \frac {k_3}{k_4}}\right\}_{0<C<C^*}$ where
\begin{equation}\label{eq:C*Htilde}
C^*=\frac {k_4}{k_4+k_3} \left( \frac {k_3}{k_4+k_3}\right)^{\frac {k_3}{k_4}}.
\end{equation}
Furthermore, every leaf with $0<C<C^*$ intersects the segment $\mathcal{R}_{py}$ at exactly two points, see Figure \ref{levelsurfaces2}(a). The value $C=C^*$ leads to a unique intersection point with coordinates $x=k_3/(k_3+k_4)$ and $z=k_4/(k_3+k_4).$ Since in the face $\mathcal{Y}$ we have $y=0,$ it follows that the point corresponding to $C^*$ is the point $\mathbf{q}_{py}$ defined in (\ref{ptoslimite}).

Similarly, the restriction of $V,\widetilde{V}$ and $H$ to the faces $\mathcal{X},\mathcal{Z}$ and $\Sigma$ respectively, are first integrals even when $\mathbf{k} \not \in \mathcal{S},$ see expression (\ref{xporhyv}). Consider the changes of variables $(u,v,\alpha,\beta) \to (y,z,k_2,k_3), (y,x,-k_1,-k_4)$ or $(x,y,k_1,k_2),$ depending on the face $\mathcal{X},$  $\mathcal{Z}$ or $\Sigma$ we are looking at. Under the assumption $\alpha \beta >0,$ the level curves $V|_{\mathcal{X}}=C^{k_2}, \widetilde{V}|_{\mathcal{Z}}=C^{k_1}$ and $H|_{\Sigma}=C^{k_1}$ define a foliation on the corresponding face, whose leaves are given by the unimodal curves $\left\{v=1-u-C u^{-\frac {\beta}{\alpha}}\right\}_{0<C<C^*}$ where 
\[
C^*=\frac {\alpha}{\alpha + \beta} \left( \frac {\beta}{\alpha + \beta} \right)^{\frac {\beta}{\alpha}}. 
\]
Every leaf with $0<C<C^*$ intersects the segment $\{v=0,0<u<1\}$ at exactly two points, see Figure \ref{levelsurfaces2}(b).  The value $C=C^*$ leads to a unique intersection point $\left({\beta}/(\alpha+\beta),0\right).$ Going back through the change of variables and adding the variable which does not appear in such change, that intersection point coincides with $\mathbf{q}_{xz},$ $\mathbf{p}_{xz}$ or $\mathbf{p}_{py}$ depending on the change of variables.

\begin{figure}
\begin{center}
\includegraphics{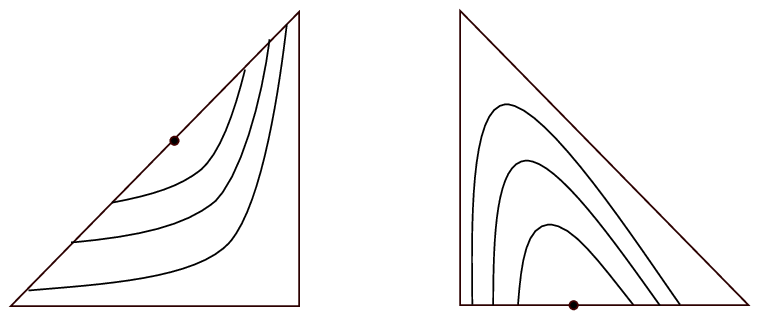}\vspace{-0.5cm}
\end{center}
\begin{picture}(0,0)(0.7,-0.1)
\put(5.6,0){(a)}
\put(4.4,0.25){$x$}
\put(7.4,3.25){$z$}
\put(3.9,0.8){$C_1$}
\put(4.5,1.35){$C_2$}
\put(4.9,1.8){$C_3$}
\put(5.4,2.3){$C^*$}
\put(10.1,0){(b)}
\put(8.6,3.25){$v$}
\put(11.6,0.25){$u$}
\put(10.1,2.75){$C_1$}
\put(10.25,2.7){\line(-1,0){0.2} \vector(-2,-1){0.4}}
\put(8.35,2.6){$C_2$}
\put(8.5,2.5){\line(1,0){0.5} \vector(1,-1){0.4}}
\put(11.1,1.75){$C_3$}
\put(11.2,1.7){\line(-1,0){0.6} \vector(-2,-1){0.6}}
\put(8.35,1.25){$C^*$}
\put(8.5,1.2){\line(1,0){1} \vector(1,-1){0.45}}
\end{picture}
\caption{\label{levelsurfaces2} (a) Foliation over the face $\mathcal{Y}$ defined by the level curves  $\widetilde{H}|_{\mathcal{Y}}=C^{k_4}$ where $0<C_1<C_2<C_3<C^*$ and $\mathbf{k}$ in $\mathcal{PS}\setminus \mathcal{S}.$ (b) Foliation over the corresponding face of the level curves  $H_{\Sigma}=C^{k_1}$ or $V_{\mathcal{X}}=C^{k_2}$ or $\widetilde{V}_{\mathcal{Z}}=C^{k_1}$ where $0<C_1<C_2<C_3<C^*$ and $\mathbf{k}$ in $\mathcal{PS}\setminus \mathcal{S}.$ Note that figure (b) is represented in $(u,v)$--coordinates.}\vspace{-0.25cm}
\end{figure}

Using the geometric information of the aforementioned foliation, in the next result we summarise the behaviour of the flow of system (\ref{3DLV}) at the boundary $\partial \mathcal{T}$ for $\mathbf{k} \in \mathcal{PS}.$

\begin{lemma}\label{lema:boundary}
\begin{itemize}
\item [(a)] If $\mathbf{k}\in \mathcal{PS},$ then each of the two 
limit sets of every orbit contained in $\mathcal{Y} \cup \Sigma$ (respectively, $\mathcal{X} \cup \mathcal{Z}$) is formed by a singular point contained in the edge $\mathcal{R}_{py}$ (respectively, in the edge $\mathcal{R}_{xz}$).
\item [(b)] If $\mathbf{k}\in \mathcal{PS} \cap \mathcal{S},$ then for every pair of orbits $\gamma_1 \subset \mathcal{Y}$ and $\gamma_2 \subset \Sigma$ (respectively, $\gamma_1 \subset \mathcal{X}$ and $\gamma_2 \subset \mathcal{Z}$) satisfying that  $\omega(\gamma_1)=\alpha(\gamma_2),$ it follows that $\alpha(\gamma_1)=\omega(\gamma_2).$ 
\item [(c)] If $\mathbf{k}\in \mathcal{PS} \setminus \mathcal{S},$ then for every pair of orbits $\gamma_1 \subset \mathcal{Y}$ and $\gamma_2 \subset \Sigma$ (respectively, $\gamma_1 \subset \mathcal{X}$ and $\gamma_2 \subset \mathcal{Z}$) satisfying that  $\omega(\gamma_1)=\alpha(\gamma_2),$ it follows that $\alpha(\gamma_1)\neq \omega(\gamma_2).$ 
\end{itemize}
\end{lemma}
\begin{proof}
We restrict ourselves to consider orbits in the faces $\mathcal{Y} \cup \Sigma.$ The study of the orbits in the faces $\mathcal{X} \cup \mathcal{Z}$ follows in a similar way.

(a) Suppose that $\mathbf{k}\in \mathcal{PS}.$ Hence $k_3k_4>0.$ Therefore every orbit $\gamma_1$ in $\mathcal{Y}$ is contained in a leaf of the foliation $z=C x^{-\frac {k_3}{k_4}}$ with $0<C<C^*,$ which is an arc of a hyperbola intersecting the edge $\mathcal{R}_{py}$ at exactly two points. Since there are not other singular points in $\mathcal{Y},$ see Proposition \ref{ptsing}(a), we conclude that each of the two limit sets of $\gamma_1$ is one of these intersection points. 

On the other hand we have $k_2k_1>0.$ Therefore every orbit $\gamma_2$ in $\Sigma$ is contained in a leaf of the foliation $y=1-x-C x^{-\frac {k_2}{k_1}},$ which is an unimodal curve intersecting the edge $\mathcal{R}_{py}$ at exactly two points. We conclude again that each of the limit sets of $\gamma_2$ is one of these intersection points. 

(b, c) Taking into account that $\Sigma$ is given by the relation $z=1-x-y,$ we express the leaves in $\Sigma$ as a function $z(x)$ in the following way $z=C x^{-\frac{k_2}{k_1}}.$ 

Let $\mathbf{p}=(x_0,0,1-x_0)$ be a point in the edge ${R}_{py}.$ There exist two positive values $C_1$ and $C_2$ such that both the leaf $z=C_1  x^{-\frac {k_3}{k_4}}$ in the face $\mathcal{Y}$ and the leaf $z=C_2  x^{-\frac {k_2}{k_1}}$ in the face $\Sigma$ 
contain the point $\mathbf{p}.$ On the other hand the leaf in the face $\mathcal{Y}$ intersects $\mathcal{R}_{py}$ at a new point $(x_1,0,1-x_1)$ and the leaf in the face $\Sigma$ intersects $\mathcal{R}_{py}$ at a new point $(x_2,0,1-x_2).$ Since two hyperbolas either intersect at most at one point or they coincide, we conclude that $k_3k_1=k_4k_2$ if and only if $x_1=x_2.$ 
\end{proof}

\section{Behaviour in the interior}

In this last section we deal with the proof of the main theorems of the paper. The next result is a technical lemma which describes planar flows under integrability conditions.

\begin{lemma}\label{integrableflows}
Let $\dot{\mathbf{x}}=\mathbf{f(x)}$ be a planar system of differential equations and let $U$ be a flow--invariant region in $\mathbb{R}^2.$ Assume that both the boundary of $U$ is formed by a heteroclinic loop and there exists exactly one singular point $\mathbf{p}$ contained in the interior of $U.$ If there exists a first integral $H$  defined over $U$ which is non--constant over open sets, then every orbit in the interior of $U$ but the singular point is a periodic orbit.
\end{lemma}
\begin{proof}
Let $\mathring{U}$ denote the interior of the region $U.$ It is easy to check that any orbit $\gamma$ in $\mathring{U}\setminus\{\mathbf{p}\}$ has its limit sets contained in the boundary of $\mathring{U}\setminus\{\mathbf{p}\}.$ Otherwise first integral $H$ would be constant over one of the open regions limited by $\gamma$ or over the whole open region $\mathring{U}\setminus\{\mathbf{p}\}.$ Hence there are not homoclinic orbits to the singular point $\mathbf{p}.$ From the Poincar\'e--Bendixson Theorem, at least one limit set of $\gamma$ is a periodic orbit $\Gamma_1$ surrounding $\mathbf{p}$ and contained in $\mathring{U}\setminus\{\mathbf{p}\}.$ 

Applying now similar arguments to the other limit set, it follows that this limit set is also a periodic orbit $\Gamma_2$ contained in $\mathring{U}\setminus\{\mathbf{p}\}$ and surrounding $\mathbf{p}.$ Since $H$ is not constant over open sets we conclude that $\Gamma_1$ and $\Gamma_2$ are the same periodic orbit and this periodic orbit coincides with $\gamma.$ Therefore every orbit in $\mathring{U}\setminus\{\mathbf{p}\}$ is a periodic orbit.
\end{proof}

\textsc{Proof of Theorem \ref{A}:} 
(a) Under the assumption $\mathbf{k} \in \mathcal{PS} \cap \mathcal{S},$ system (\ref{3DLV}) is integrable and the functions $H$ and $V$ are two independent first integrals, see Proposition \ref{firstint}(a). Since any level surface $\mathcal{H}_C$  is invariant by the flow, we can consider the restriction of the flow to each of these  surfaces. Of course this restricted flow is also integrable because the restriction of the function $V$ to $\mathcal{H}_C$ is a first integral. On the other hand there exists exactly one singular point in  the interior of $\mathcal{H}_C,$ which comes from the intersection of the manifold $\mathcal{H}_C$ and the segment $R$ defined in the Proposition \ref{ptsing}(a-1). 

The map $\pi(x,y,z)=(x,y)$ projects the manifold $\mathcal{H}_C$ over a compact region $U.$ Moreover the Jacobian matrix of $\pi$ at the point $\pi^{-1}(x,y)$ defines a flow over $U$ which is differentially conjugate to the restricted flow over  $\mathcal{H}_C.$ Since there exists exactly one singular point in the interior of $U,$ the function $V\circ\pi^{-1}$ is a first integral over $U$ which is non--constant over open sets and the boundary of $U$ is formed by a heteroclinic loop (see Lemma \ref{lema:boundary}(b)), from Lemma \ref{integrableflows} it follows that every orbit in $U$ but the singular point is a periodic orbit. Therefore, every orbit over $\mathcal{H}_C$ but the singular point is a periodic orbit. This result is independent on the level surface we are working at, hence  every orbit in the interior of the region $\mathcal{T},$ but the singular points, is a periodic orbit. 

The behaviour of the flow at the boundary of $\mathcal{T}$ when $\mathbf{k} \in \mathcal{PS} \cap \mathcal{S}$ can be obtained from Lemma \ref{lema:boundary}(b).

(b) Consider now that $\mathbf{k} \not \in \mathcal{PS} \cap \mathcal{S}$ and $\mathbf{k} \neq \mathbf{0}.$ We distinguish between two situations: first we suppose that $\mathbf{k} \in \mathcal{S}\setminus \mathcal{PS}.$ In such case $\mathbf{k}$ belongs to the manifold $\mathcal{S}.$ From Proposition \ref{firstint} it follows that at least one of the functions $H,V,\widetilde{H}$ or $\widetilde{V}$ is a first integral. Without loss of generality we can assume that $H$ is a first integral. Hence any level surface $\mathcal{H}_C$ is invariant by the flow and we can consider the restriction of the flow to $\mathcal{H}_C.$ From Proposition \ref{ptsing} there are not singular points in the interior of $\mathcal{H}_C.$  Applying the Poincar\'e--Bendixson Theorem to the flow in the level surface $\mathcal{H}_C$, we conclude that the flow goes from the boundary of $\mathcal{H}_C$ to the boundary of $\mathcal{H}_C.$ Since these arguments are independent on the level surface, it follows that the limit sets of every orbit in the interior of $\mathcal{T}$ is contained in $\partial \mathcal{T}.$

Suppose now that $\mathbf{k} \not \in \mathcal{S}$ and $\mathbf{k} \neq \mathbf{0}.$ Since one of the coordinates of $\mathbf{k}$ is different from zero, the level surfaces of at least one of the functions $H,V,\widetilde{H}$ and $\widetilde{V}$ can be expressed as the graph of an explicit differentiable function. For instance if $k_4\neq 0,$ then $\widetilde{\mathcal{H}}_{C^{k_4}}$ is the graph of the function $z=C x^{-\frac {k_3}{k_4}}$ defined over the face $\mathcal{Z}.$ 
Each of these level surfaces split the interior of $\mathcal{T}$ into two disjoint connected components. On the other hand since  $\mathbf{k} \not \in \mathcal{S}$ these level surfaces are not invariant by the flow, see Proposition \ref{firstint}(c). In fact the flow is transversal to them and the direction of the flow through them depends on $\mathbf{k} \in \mathcal{S}^+$ or $\mathbf{k} \in \mathcal{S}^-,$ see expression (\ref{xporhyv}). 
Since as $C$ tends to $0$ or to $C^*$ the level surfaces $\widetilde{\mathcal{H}}_{C^{k_4}}$ tend to the boundary of $\mathcal{T},$ we conclude that the flow in the interior of $\mathcal{T}$ goes from one part of the boundary to another part of the boundary. That is the limit sets of every orbit in the interior of $\mathcal{T}$ are contained in $\partial \mathcal{T}.$ 

Thus, in both cases the limit sets of every orbit in the interior of $\mathcal{T}$ are contained in $\partial \mathcal{T}.$ Since there are no isolated singular points in $\partial \mathcal{T},$ see Proposition \ref{ptsing}, we conclude that these limit sets are not periodic orbits. $\hfill \blacksquare$\newline

Note that in the previous proof we have only used that trajectories cross the level surfaces of some of the functions $H,V,\widetilde{H}$ or $\widetilde{V},$ always in the same direction. This suffices to conclude that the limit sets of the orbits in $\mathcal{T}$ are contained in the boundary. To prove Theorem \ref{B} we need to be more precise in the location of these limit sets. To reach this goal we will control the geometry of the level surfaces.\newline

\begin{figure}
\begin{center}
\includegraphics{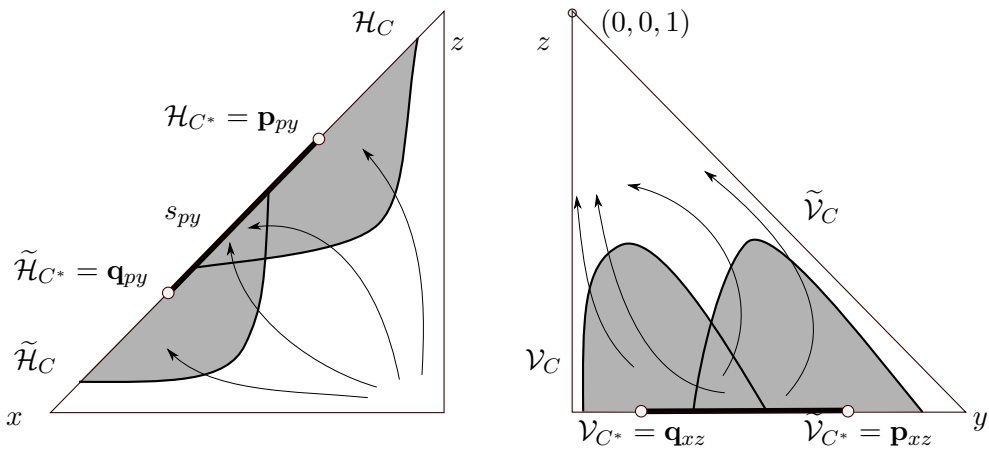}
\caption{\label{theo2}Representation when $\mathbf{k} \in \mathcal{PS}_+ \cap \mathcal{S}^+$
of: the positive invariant regions (in grey) limited by the level surfaces $\mathcal{H}_C$ and $\widetilde{\mathcal{H}}_C;$ the negative invariant regions (in grey) limited by the level surfaces $\mathcal{V}_C$ and $\widetilde{\mathcal{V}}_C;$ the segment $s_{py}$ formed by the $\omega$--limit set of each orbit in the interior of $\mathcal{T};$ and the segment $s_{xz}$ formed by the $\alpha$--limit set of each orbit in the interior of $\mathcal{T}.$}
\begin{picture}(0,0)(0.075,-0.5)
\put(-4.5,5){$s_{py}$}
\put(-4.5,6.25){$\mathcal{H}_{C^*}=\mathbf{p}_{py}$}
\put(-6.5,4.2){$\widetilde{\mathcal{H}}_{C^*}=\mathbf{q}_{py}$}
\put(-6.6,2.25){$x$}
\put(-0.7,7,25){$z$}
\put(-6.5,3){$\widetilde{\mathcal{H}}_C$}
\put(-2,7.5){${\mathcal{H}}_C$}
\put(0.45,7.25){$z$}
\put(6.25,2.25){$y$}
\put(0.3,3){$\mathcal{V}_C$}
\put(4,5){$\widetilde{\mathcal{V}}_C$}
\put(1,2.05){$\mathcal{V}_{C^*}=\mathbf{q}_{xz}$}
\put(4,2.05){$\widetilde{\mathcal{V}}_{C^*}=\mathbf{p}_{xz}$}
\put(0.93,7.74){\circle{0.12}}
\put(1.3,7.5){$(0,0,1)$}
\end{picture}\vspace{-1.5cm}
\end{center}
\end{figure}

\textsc{Proof of Theorem \ref{B}:}
(a) Suppose that $\mathbf{k} \in \mathcal{PS}_+ \cap \mathcal{S}^+.$  Since $\mathbf{k} \not \in \mathcal{S}$ the functions $H$  and $\widetilde{H}$ are not first integrals and each level surface $\mathcal{H}_C$ and $\widetilde{\mathcal{H}}_C$ splits the region $\mathcal{T}$ into two disjoint regions in such a way that the flow goes from one to the other. In fact since $k_3k_4>0$ the intersection of $\widetilde{\mathcal{H}}_C$ with any plane $\{y=y_0\},$ where $0<y_0<1,$ is an arc of a hyperbola in the $(x,z)$--plane, see Figure \ref{theo2}. Similarly, since $k_1k_2>0$ the intersection of $\mathcal{H}_C$ with any plane $\{y=y_0\},$ where $0<y_0<1,$ is an arc of a hyperbola in the $(x,z)$--plane. The flow through $\mathcal{H}_C$ and through $\widetilde{\mathcal{H}}_C$ has the same orientation as the vectors $\nabla H$ and $\nabla  \widetilde{H}$ respectively, see expression (\ref{xporhyv}). Since $k_i>0$ the coordinates of the gradient $\nabla H$ are non--negatives. Hence $\nabla H$ is oriented towards the region containing the point $\mathbf{p}_{py},$ see the shadowed region in Figure \ref{theo2}. In a similar way, the gradient $\nabla \widetilde{H}$ is oriented towards the region containing the point $\mathbf{q}_{py},$ see Figure \ref{theo2}. Therefore the flow evolves from the region containing the origin to the region containing the segment $s_{py}$, see Figure \ref{theo2}. On the other hand points in $\mathcal{R}_{py}\setminus s_{py}$ are not limit set of orbits in the interior of $\mathcal{T},$ see Proposition \ref{prop:ptssing}. We conclude that the $\omega$--limit set of any given orbit in the interior of $\mathcal{T}$ is a singular point contained in the segment $s_{py}.$

Since $\mathbf{k} \not \in \mathcal{S}$ the functions $V$  and $\widetilde{V}$ are not first integrals. Moreover each level surface $\mathcal{V}_C$ and $\widetilde{\mathcal{V}}_C$ splits the region $\mathcal{T}$ into two disjoint regions in such a way that the flow goes from one to the other. The flow through these surfaces has opposite direction to that of the gradients $\nabla V$ and $\nabla  \widetilde{V},$ see expression (\ref{xporhyv}). Since $k_i>0$ the gradient $\nabla V$ is oriented towards the region containing the point $\mathbf{q}_{xz}$ and the gradient $\nabla \widetilde{V}$ is oriented towards the region containing the point $\mathbf{p}_{xz}.$ Therefore the flow in the interior of $\mathcal{T}$ evolves from the shadowed region in Figure \ref{theo2} towards the region containing the point $(0,0,1).$ Since points in $\mathcal{R}_{xz}\setminus s_{xz}$ are not limit set of orbits in the interior of $\mathcal{T},$ see Proposition \ref{prop:ptssing}, we conclude that the $\alpha$--limit set of any given orbit in the interior of $\mathcal{T}$ is a singular point contained in the segment $s_{xz},$  see Figure \ref{theo2}.

As we have just proved when $\mathbf{k} \in \mathcal{PS}_+ \cap \mathcal{S}^+$ the $\omega$--limit set and the $\alpha$--limit set of any given orbit in the interior of $\mathcal{T}$ is contained in the segments $s_{py}$ and $s_{xz},$ respectively. Similar arguments apply when $\mathbf{k} \in \mathcal{PS}_- \cap \mathcal{S}^-.$ 

Note that a change of the sign of the parameter $\mathbf{k}$ is equivalent to a change in the sign of time in the system of differential equations (\ref{3DLV}). Therefore the behaviour of the flow in cases $\mathbf{k} \in \mathcal{PS}_+ \cap \mathcal{S}^-$ and $\mathbf{k} \in \mathcal{PS}_- \cap \mathcal{S}^+$ follows from the cases described above by changing the orientation of the orbits.

(b,c) The behaviour of the flow at the boundary can be obtained from Lemma \ref{lema:boundary}(c).
$\hfill\blacksquare$

\section*{Acknowledgements}
AET is partially supported by MEC grant number MTM2005-06098-C02-1,
by UIB grant number UIB2005/6 and by CAIB grant number CEH-064864. ACM acknowledges support from the BioSim Network, grant number LSHB-CT-2004-005137. We thank the referees for their careful reading of our manuscript and their useful comments on the presentation of this article.



\end{document}